\documentclass[11pt]{amsart}
\usepackage{amscd, amsmath}
\usepackage{ amssymb}
\usepackage[margin=3.2cm]{geometry}

\newtheorem{theorem}{Theorem}[section]
\newtheorem{lemma}[theorem]{Lemma}
\newtheorem{cor}[theorem]{Corollary}
\theoremstyle{definition}

\theoremstyle{remark}
\newtheorem{remark}[theorem]{Remark}
\numberwithin{equation}{subsection}
\theoremstyle{plain}
\newtheorem*{ack}{Acknowledgement}

\newtheorem{question}{Question}

\newtheorem*{lemmau}{Lemma}

\def\F{\mathbb F}

\def\V{\mathbb V}
\def\W{\mathbb W}

\def\F{\mathbb F}

\def\V{\mathbb V}
\def\W{\mathbb W}

\copyrightinfo{2001}{enter name of copyright holder}
\newcommand{\secref}[1]{section~\ref{#1}}
\newcommand{\thmref}[1]{Theorem~\ref{#1}}
\newcommand{\lemref}[1]{Lemma~\ref{#1}}

\newcommand{\eqnref}[1]{~{\textrm(\ref{#1})}}

\begin{document}
 \title[Invariant form under a linear map]{Existence of an invariant  form under  a linear map}
\author[Krishnendu Gongopadhyay \and Sudip Mazumder]{Krishnendu Gongopadhyay \and Sudip Mazumder}
\address{Indian Institute of Science Education and Research (IISER) Mohali, Sector 81, S.A.S. Nagar, Punjab 140306, India}
\email{krishnendug@gmail.com, krishnendu@iisermohali.ac.in}
\address{Department of Mathematics, Jadavpur University, Jadavpur, Kolkata 700032 }
\email{topologysudip@gmail.com}
\begin{abstract}

Let $\F$ be a field of characteristic different from $2$ and $\V$ be a vector space over $\F$.  Let $J: \alpha \to \alpha^J$ be a fixed involutory automorphism on $\F$. 
 In this paper we answer the following question: given an invertible linear map $T: \V \to \V$, when does the vector space $\V$ admit a $T$-invariant non-degenerate $J$-hermitian, resp. $J$-skew-hermitian, form? 
\end{abstract}

\maketitle

\section{Introduction}\label{intro} 
Let $\F$ be a  field of characteristic different from $2$.  Let $J: \alpha \to \alpha^J$ be a fixed involutory automorphism on $\F$, i.e.  $(\alpha+ \beta)^J= \alpha^J + \beta^J$, $(\alpha \beta)^J= \alpha^J \beta^J$, $(\alpha^{J})^J=\alpha$.  If there exists a non-zero $\alpha$ such that $J(\alpha)=\alpha^J=-\alpha$, we say:  ``$x^J=-x$ has a solution in $\F$".  This is always the case if $J$ is non-trivial.  Let $\V$ be a finite dimensional vector space over $\F$. In this paper we ask the following question.
\begin{question}\label{q1}
Given an invertible linear map $T: \V \to \V$, when does the vector space $\V$ over $\F$ admit a $T$-invariant non-degenerate $J$-hermitian, resp. $J$-skew-hermitian, form? 
\end{question}
We have answered the question in this note. This generalizes earlier work by  Gongopadhyay and Kulkarni \cite{gk} where the authors obtained conditions for an invertible linear map to admit an invariant non-degenerate quadratic and symplectic form assuming that the underlying field is of large characteristic. de Seguins Pazzis \cite{clement} extended the work of \cite{gk} over arbitrary characteristic.  The technicalities are slightly different in these works due to the underlying field characteristic. In this work we follow ideas from \cite{gk}.

\medskip  Let $\overline \F$ be the algebraic closure of $\F$. 
Let $f(x)=\sum_{i=0}^d a_i x^i$, $a_d=1$, be a monic polynomial of degree $d$ over $\F$ such that $-1$, $0$, $1$ are not its roots. 
The \emph{dual} of $f(x)$ is defined to be the polynomial $f^{\ast}(x)=(f(0)^J)^{-1}x^d f^J(x^{-1})$, where $f^J(x)=\sum_{i=0}^d a_i^J x^i$.  Thus,  $f^{\ast}(x)=\frac{1}{a_0^J} \Sigma_{i=0}^d a_{d-i}^J x^{i}$. In other words, if $\alpha$ in $\overline \F$ is a root of $f(x)$ with multiplicity $k$, then $(\alpha^J)^{-1}$ is a root of $f^{\ast}(x)$ with the same multiplicity. The polynomial $f(x)$ is said to be \emph{self-dual} if $f(x)=f^{\ast}(x)$. 

Let $T: \V \to \V$ be a linear transformation.  A $T$-invariant subspace is said to be \emph{indecomposable} with respect to $T$, or simply \emph{$T$-indecomposable} if it can not be expressed as a direct sum of two proper $T$-invariant subspaces.  $\V$ can be written as a direct sum $\V=\oplus_{i=1}^m \V_i$, where each $\V_i$ is $T$-indecomposable for $i=1,2, \ldots, m$. In general, this decomposition is not canonical. But for each $i$, $(\V_i, T|_{\V_i})$ is ``dynamically equivalent'' to $(\F[x]/(p(x)^k), \mu_x)$, where $p(x)$ is an irreducible monic factor of the minimal polynomial of $T$, and $\mu_x$ is the operator 
$[u(x)] \mapsto [xu(x)]$, for eg. see \cite{kul}.  Such $p(x)^k$ is an \emph{elementary divisor} of $T$. If $p(x)^k$ occurs $d$ times in the decomposition, we call $d$ the \emph{multiplicity} of the elementary divisor $p(x)^k$. 

\medskip Let $\chi_T(x)$ denote the characteristic polynomial of an invertible linear map $T$. Let
 $$\chi_T(x)=(x-1)^e (x+1)^f \chi_{oT}(x),$$ where $e, f \geq 0$, and $\chi_{oT}(x)$ has no root $1$, or $-1$.  The vector space $\V$ has a $T$-invariant decomposition $\V=\V_1 \oplus \V_{-1}\oplus \V_o$, where for $\lambda=1, -1$, $\V_{\lambda}$ is the generalized eigenspace to $\lambda$, i.e.
$$\V_{\lambda}=\{v \in \V\;|\;(T-\lambda I)^n v=0\}, $$
and $\V_o=ker \ \chi_{oT}(T)$.  Let $T_o$ denote the restriction of $T$ to $\V_o$. Clearly $T_o$ has the characteristic polynomial $\chi_{oT}(x)$ and  does not have any eigenvalue 1 or -1.

\medskip 
With the notations as given above, we prove the following theorem that answers the above question.  When  $J$ is the trivial automorphism, the following theorem descends to Theorem 1.1 of \cite{gk} in view of \lemref{cuni} in \secref{mainth}. 
\begin{theorem}\label{ansq1}   Let $\F$ be a field with characteristic different from two. Let $J: \alpha \to \alpha^J$ be a fixed non-trivial involutory automorphism on $\F$. Let $\V$ be a vector space over $\F$ of dimension at least $2$. Let $T: \V \to \V$ be an invertible linear map.  Then $\V$ admits a $T$-invariant non-degenerate $J$-hermitian, resp. $J$-skew-hermitian form if and only if  an elementary divisor of $T_o$ is either self-dual, or its dual is also an elementary divisor with the same multiplicity.\end{theorem}
We prove this theorem in the next section.

\medskip Suppose $T: \V \to \V$ is a linear map that admits  an invariant hermitian or skew-hermitian form $H$. Then canonical forms for $T$ is known in the literature, for example, see \cite{wall}. This provides the necessary condition in the above theorem. The focus of this article is the converse part, i.e. the sufficient conditions for a linear map $T$ to admit a non-degenerate hermitian form. 

After finishing this work, we found the papers by Sergeichuk \cite[Theorem 5]{s1}, \cite[Theorem 2.2]{s2}, where the author has obtained canonical forms for the pairs $(A, B)$, where $B$ is a non-degenerate form and $A$ is an isometry of $B$ over a field of characteristic not $2$. The work of Sergeichuk not only gives the necessary condition stated above, but the sufficient condition is also implicit there. However, it has not been stated in a precise form as in the above theorem. Sergeichuk's approach involves quivers to derive the results. Our approach is simpler here.

\section{Proof of \thmref{ansq1}} \label{mainth}
\medskip For a matrix $A=(a_{ij})$, let $A^J$ be the matrix $A^J=(a_{ij}^J)$. 
  Let $H$ be a $J$-sesquilinear form on $\V$. The form $H$ is \emph{$J$-hermitian} or simply \emph{hermitian}, resp. \emph{skew-hermitian}  if for all $u, v$ in $\V$, $H(u, v)=H(v, u)^J$, resp.  $H(u, v)=-H(v, u)^J$. For a linear map $T: \V \to \V$, we say $H$ is $T$-invariant or $T$ is an isometry of $H$ if for all $u, v \in \V$, $H(Tu, Tv)=H(u, v)$.

\begin{lemma}\label{cuni}
 Let $T: \V \to \V$ be a unipotent linear map with minimal polynomial $(x-1)^n$, $n \geq 2$. Let $\V$ be $T$-indecomposable.  
\begin{itemize} 
\item[(i)] If  $n$ is even, resp. odd, then  $\V$ admits a $T$-invariant skew-hermitian, resp. 
hermitian form.

\item[(ii)] 
If $n$ is even, resp. odd,  and  $x^J=-x$ has a non-zero solution in $\F$, then  $\V$ admits a $T$-invariant hermitian, resp. skew-hermitian form.
\end{itemize} 
\end{lemma}

\begin{proof} 

Let $T$ be an unipotent linear map. Suppose the minimal polynomial of $T$ is $m_T(x)=(x-1)^n$. Without loss of generality we can assume that $T$ is of the form
\begin{equation}\label{t}
T=\begin{pmatrix} 1 & 0& 0& 0 & .... \ 0 & 0\\ 1 & 1 & 0& 0 & .... \ 0& 0 \\ 0 & 1 & 1 & 0 & ....\ 0& 0\\ &\ddots & \ddots  & \ddots    &\ddots& \\ 0 &0 & 0 & \ldots & 1&0\\ 0 & 0 & 0 & \ldots & 1& 1 \end{pmatrix}
\end{equation}
Suppose $T$ preserves a $J$-sesquilinear form $H$. In matrix form, let $H=(a_{ij})$.
Then,  $(T^{J})^{t} H T =H$.
This gives the following relations: For $1 \leq i \leq n-1$,
\begin{equation}\label{1}
 a_{i+1,n}=0=a_{n, i+1},
\end{equation}
\begin{equation}\label{2'}
a_{i,j} + a_{i,j+1}+a_{i+1,j}+a_{i+1,j+1}=a_{i,j},
\end{equation}
\begin{equation}\label{2}
 \hbox{i.e. } \  a_{i,j+1}+a_{i+1,j}+a_{i+1,j+1}=0. \end{equation}
From the above two equations we have, for $1 \leq l \leq n-3$ and $l+2 \leq i \leq n-1$,
\begin{equation}\label{3}
a_{i,n-l}=0=a_{n-l,i}.
\end{equation}
This implies that $H$ is a  matrix of the form
\begin{equation}\label{can}
H= \begin{pmatrix} a_{1,1} & a_{1,2} & a_{1,3} & a_{1,4} & .... & a_{1, n-2} & a_{1,n-1} & a_{1,n}\\
a_{2,1} & a_{2,2} & a_{2,3} & a_{2,4} & .... & a_{2,n-2} & a_{2,n-1} & 0\\
a_{3,1} & a_{3,2} & a_{3,3} & a_{3,4} & .... & a_{3,n-2} & 0 & 0\\
\vdots& \vdots&\vdots  &  \vdots& .... &\vdots  & \vdots & \vdots \\
a_{n-1,1} & a_{n-1,2} & 0 & 0 & ....& 0 & 0 & 0 \\
a_{n,1} & 0 & 0 & 0 & ....& 0& 0 & 0 \end{pmatrix},
\end{equation}
where, 

\begin{equation*} \label{e1} a_{i+1, j} + a_{i,  j+1} + a_{i+1, j+1}=0. 
\end{equation*}
 Choose a basis $ {e_{1},e_{2},e_{3},...,e_{n}}$ of $ V $
 such that $ T $ and $ H $ has the above forms  with respect to the basis.
  From \eqnref{2} we have
 \begin{equation}\label{e1} a_{l,n-l+1}=(-1)^{n+1-2l} \ a_{n-l+1,l}.\end{equation} 
Note that it follows from \eqnref{2} that $a_{l, n-l+1}=(-1)^{l-1} a_{1, n}$ for $1 \leq l \leq n$. Hence $H$ is non-singular if and only if $a_{1, n} \neq 0$. Continuing the procedure, all entries of $H$ except $a_{1,1}$ can be expressed as a scalar multiple of $a_{1, n}$.

 For $H$ to be hermitian, from \eqnref{e1} we must have 
$ a_{l,n-l+1}^J=a_{n-l+1,l}=(-1)^{n+1}a_{l, n-l+1}$. This implies, $a_{1,n}^J=(-1)^{n+1}a_{1,n}$. So, a non-zero choice of $a_{1,n}$ is possible only if either $n$ is odd or, in case $n$ is even, then $ x^{J}=-x $ must have a solution in $ \F $, which is always the case. The other case is similar. 
\end{proof}

\begin{cor}
Let $\F$ be a field of characteristic different from $2$ such that $x^J=-x$ has a solution in $\F$. 
Let $\V$ be an $n$-dimensional vector space of dimension $\geq 2$ over $\F$. Let $T: \V \to \V$ be a unipotent linear map such that $\V$ is $T$-indecomposable.  Then $\V$ admits a $T$-invariant non-degenerate hermitian, as well as skew-hermitian form. 
\end{cor} 

\begin{lemma}\label{qsf}
Let $T:\V \to \V$ be an invertible linear map. If $T$ has no eigenvalue $ 1$ or $-1$, then there exists a $T$-invariant non-degenerate hermitian form on $\V$ if and only if there exists a $T$-invariant non-degenerate skew-hermitian form on $\V$. 
\end{lemma}
\begin{proof} Assume, $H$ is a $T$-invariant hermitian form. 
Define a form $H_T$ on $\V$ as follows:
$$\hbox{For }u, v \hbox{ in } \V, \;\;H_T(u, v)=H((T-T^{-1})u, v).$$ 
Note that 
\begin{eqnarray*}
H_T(u, v)&=&H((T-T^{-1})u, v)\\
&=&H(Tu, v)-H(T^{-1}u, v)\\
&=&H(u, T^{-1}v)-H(u, Tv), \hbox{ since $T$ is an isometry.}\\
&=&H(u, T^{-1}v-Tv)\\
&=&-H(u, (T-T^{-1})v)\\
&=&-H_T^J((T-T^{-1})v, u), \hbox{ since $H$ is hermitian.}\\
&=&-H^J_T(v,u).
\end{eqnarray*}
Thus $H_T$ is a $T$-invariant  non-degenerate skew-hermitian form on $\V$.  Also it follows by the same construction that corresponding to each $T$-invariant skew-hermitian form, there is a canonical $T$-invariant hermitian form.
\end{proof}
\begin{lemma}\label{p-inv1}
 Let $T: \V \to \V$ be an invertible  linear map with characteristic polynomial $\chi_T(x)=p(x)^d$, where $p(x) \neq x \pm 1$, is irreducible over $\F$, and is self-dual. Let $\V$ be $T$-indecomposable. Then there exists a $T$-invariant non-degenerate hermitian, resp. skew-hermitian form on $\V$. 
\end{lemma}

\begin{proof}
Since $\V$ is $T$-indecomposable, $(\V, T)$ is dynamically equivalent to the pair $\big(\F[x]/(p(x)^d)$, $\mu_x \big)$, where $\mu_x $ is the operator
$\mu_: [u(x)] \mapsto [xu(x)]$, cf.  \cite{kul}. Hence without loss of generality we assume $\V=\F[x]/(p(x)^d)$, $T=\mu_x$ and let
$$\mathcal B=\{e_1=1, e_2=x, \ldots, e_{k}=x^{k-1}\}, ~k=d. \deg p(x).$$
be the corresponding basis.

 Let 
$\chi_T(x)=p(x)^d=\sum_{i=0}^{k} c_i x^i$. Since $\chi_T(x)$ is self-dual, we must have $c_i=\frac{c_{k-i}^J}{c_0^J}$, $1 \leq i \leq k-1$, $c_0 c_0^J=1$.   If possible, suppose $H=(h_{ij})$ be a $T$-invariant sesquilinear form on $\V$. Then
$$h_{ij}=H(e_i, e_j)=H(\mu_x e_i, \mu_x e_j)=h_{i+1, j+1}, ~1 \leq i \leq k, ~ 1 \leq j \leq k.$$
Hence,  a possible $T$-invaraint sesquilinear form should be represnted necessarily by a matrix of the following form:
\begin{equation} \label{hf1} X=\begin{pmatrix} \alpha_1 & \alpha_2 & \alpha_3 & \ldots & \alpha_{k-1} & \alpha_{k} \\ \beta_2 & \alpha_1 & \alpha_2 & \ldots & \alpha_{k-2} & \alpha_{k-1} \\ \beta_3 & \beta_2 & \alpha_1 & \ldots & \alpha_{k-3}& \alpha_{k-2} \\
\vdots & \vdots & \vdots& \vdots& \vdots& \vdots \\ \beta_{k-1} & \beta_{k-2} & \beta_{k-3} & \ldots & \alpha_1 & \alpha_2    \\ \beta_{k} & \beta_{k-1} & \beta_{k-2} & \ldots & \beta_2 & \alpha_1 \end{pmatrix}.\end{equation}

Let
$$\mathcal S=\{A=(a_{ij}) \in M_{k}(\F) \ | \   (T^J)^t A T=A \}.$$
Thus $T$ has an invariant form if and only if $\mathcal S \neq \phi$.

 Let $C$ be the companion matrix of $\mu_x$ given by:
$$C=\begin{pmatrix} 0 & 0 & 0 & \ldots & 0 & -c_0\\ 1 & 0 & 0 & \ldots & 0 & -c_1\\ 0 & 1 & 0 & \ldots & 0 & -c_2 \\ \vdots & \vdots & \vdots& \vdots & \vdots \\ 0 & 0 & 0 & \ldots & 1 & -c_{k-1} \end{pmatrix}. $$

Let
$$H_1=\begin{pmatrix} \beta_1 & 0& 0 & \ldots & 0 & 0\\ \beta_2 & \beta_1 & 0 & \ldots & 0& 0 \\ \beta_3 & \beta_2 & \beta_1 & \ldots &0 & 0 \\
\vdots & \vdots & \vdots& \ldots& \vdots & \vdots \\ \beta_{k-1} & \beta_{k-2} & \beta_{k-3} & \ldots  & \beta_1& 0 \\ \beta_{k} & \beta_{k-1} & \beta_{k-2} & \ldots  & \beta_2& \beta_1 \end{pmatrix}.$$

\medskip We consider the equation $(C^J)^t H_1 C=H_1$. Simplifying the left hand side, we get
$$(C^J)^t H_1 C=\begin{pmatrix} \beta_1 & 0 & 0 & 0 &  \ldots & 0 & a_1\\
\beta_2 & \beta_1 & 0 & 0 &  \ldots & 0 & a_2 \\ \beta_3 & \beta_2 & \beta_1 & 0 &  \ldots & 0 & a_3 \\ \beta_4 & \beta_3 & \beta_2 & \beta_1 & \ldots& 0 & a_4 \\ \vdots & \vdots& \vdots& \vdots& \ldots & \vdots& \vdots\\ \beta_{k-1} & \beta_{k-2} & \beta_{k-3} & \beta_{k-4} & \ldots & \beta_1 & a_{k-1} \\ b_{k-1}& b_{k-2}& b_{k-3} & b_{k-4}& \ldots & b_{1} & -\sum_{i=0}^{k-1} c_i b_{k-i} \end{pmatrix},$$
where $$a_i= -c_0\beta_{i+1} - c_1 \beta_i - c_2 \beta_{i-1} - \ldots - c_{i-1} \beta_1, $$
$$b_{i}=-{c}^J_{k-1}\beta_{i} - {c}^J_{k-2}\beta_{i-1}  - \ldots - {c}^J_{k-i} \beta_{1}.$$
Comparing both sides of $(C^J)^t H_1 C=H_1$ gives
\begin{equation}\label{a1}  
\beta_{i+1}=\frac{1}{c_0}( - c_1 \beta_i - c_2 \beta_{i-1} - \ldots - c_{i-1} \beta_1), ~ 1 \leq i \leq k-2. 
\end{equation}
Now, by back substitution it is easy to see that all $\beta_i$ can be expressed as a multiple of $\beta_1$ by an  expression in $\frac{c_1}{c_0}, \ldots, \frac{c_{k-1}}{c_0}$. 
Next, consider
$$
H_2 = \begin{pmatrix} \alpha_1 & \alpha_2 & \alpha_3 & \ldots &\alpha_{k-1} &  \alpha_{k} \\ 0 & \alpha_1 & \alpha_2 & \ldots & \alpha_{k-2} & \alpha_{k-1} \\  0 & 0 & \alpha_1 & \ldots &\alpha_{k-3} &  \alpha_{k-2} \\
\vdots & \vdots & \vdots& \ldots& \vdots& \vdots \\ 0 & 0 & 0 & \ldots & \alpha_1& \alpha_2 \\ 0 & 0 &0 & \ldots & 0 & \alpha_1 \end{pmatrix}. $$
Comparing both sides of the equation $(C^J)^t H_2 C=H_2$ we get
\begin{equation} \label{a2} 
c_0^J\alpha_{i+1}=-c_1^J \alpha_{i} - -c^J_2 \alpha_{i-1}--c^J_3 \alpha_{i-2}- \ldots - c^J_{i-1} \alpha_{1},  ~~ 1 \leq i \leq k-1. 
\end{equation}
By  back substitution it is easy to see that each $\alpha_i$ can be expressed as a multiple of $\alpha_1$ by an  expression in $\frac{c_1^J}{c_0^J}, \ldots, \frac{c_{k-1}^J}{c_0^J}$.

Thus $H_1$ and $H_2$ are elements from the set $\mathcal S$ and for $\beta_1 \neq 0$, resp. $\alpha_1 \neq 0$, they give non-degenerate sesquilinear forms.  We also see that $H=H_1 + H_2$ is an element in the set $\mathcal S$ and is of the form \eqnref{hf1}.
If we choose $\beta_1 =\alpha_1^{J}$, it follows from \eqnref{a1} and \eqnref{a2} that $\alpha_{i+1}=\beta_{i+1}^J$, $1 \leq i \leq k-1$, and hence the form $H$ is hermitian. If we choose $\beta_1=- \alpha^{J}_1$, it follows $H$ is skew-hermitian. It can be seen  that  $H$ may be chosen to be non-degenerate. 
\end{proof}
\medskip 

\begin{remark} 
We would like to clarify a small inaccuracy in the statement of \cite[Lemma 3.2(i)]{gk}. It has been stated there that  for $T$ unipotent and $(\V, B)$ a $T$-indecomposable bilinear space,  the bilinear form $B$ degenerate implies $B=0$. This statement needs slight modification. For example, if $\V$ is of dimension $2$, 
$$T=\begin{pmatrix} 1 & 0 \\ 1 & 1 \end{pmatrix}, ~ \hbox{ and } ~~B=\begin{pmatrix} 1 & 0 \\ 0 & 0 \end{pmatrix},$$
then $\V$ is $T$-indecomposable, preserves $B$ which is degenerate, but $B \neq 0$.  Following similar computations as in  the proof of \lemref{cuni}, we can modify the statement as follows: with the hypothesis as in Lemma 3.2 of \cite{gk}, $B$ degenerate implies that the radical of $B$ is a co-dimension one subspace of $\V$. What had been used in the relevant parts of \cite{gk}, especially in the proof of Theorem 1.3., is actually the following lemma. 

\begin{lemmau}\label{unind}
Let $\V$ be a vector space equipped with a non-degenerate symmetric or skew-symmetric bilinear form $B$ over a field $\F$ of large characteristic. 
 Suppose $T: \V \to \V$ is a unipotent isometry. Let $\W$ be a $T$-indecomposable subspace of $\V$.  Then either $B|_{\W}=0$ or,  $B|_{\W}$ is non-degenerate.  
\end{lemmau}
 The proof of the lemma is a slight modification of the proof of Lemma 2.2(i) in \cite{g}. 
\end{remark} 

\medskip

\subsection{Proof of \thmref{ansq1}} \begin{proof}
 Suppose that the linear map $T$ admits an invariant non-degenerate hermitian, resp skew-hermitian,  form  $H$.
 Then the necessary condition follows from existing literatures, for example see \cite{wall, s1, s2}. 

\medskip Conversely, let  $\V $  be  a  vector  space  of  dim $  n\ge 2 $ over  the  field
$ \F $ and $ T : \V \longrightarrow \V $ an  invertible  map  such  that  the an
elementary  divisor of $ T_o $ is   either self-dual or, its dual is also an elementary divisior. For an elementary divisor $h(x)$, let $\V_h$ denote the $T$-indecomposable subspace  isomorphic to $\F[x]/(h(x))$.
From the structure theory of linear maps, for eg. see \cite{jac}, \cite{kul},  it follows that $\V$ has a primary decomposition of the form 
\begin{equation}\label{x}
\V=\bigoplus_{i=1}^{m_1}\V_{f_i} \oplus \bigoplus_{j=1}^{m_2} (\V_{g_i} \oplus \V_{g_i^{\ast}}),
\end{equation} 
where  for each $i=1,2,...,m_1$,   $f_i(x)$ is either self-dual, or one of $(x+1)^k$ and $(x-1)^k$, for each $j=1,2,...,m_2$, $g_i(x)$, $g_i^{\ast}(x)$ are dual to each other and $g_i(x)\neq g_i^{\ast}(x)$.  To prove the theorem it is sufficient to induce a $T$-invariant hermitian (resp. skew-hermitian) form on each of the summands. 

Let $\W$ be an $T$-indecomposable summand in the above decomposition and let $p(x)^k$ be the corresponding elementary divisor. Suppose $p(x)^k$ is self-dual. It follows from \lemref{p-inv1} that there exists a $T$-invariant non-degenerate hermitian, as well as skew-hermitian form on $\W$. 

  Suppose $p(x)^k$ is not self-dual. Then there is a dual elementary divisor $p^{\ast}(x)^k$. $\W_p=\hbox{ker }p(T)^k$, $\W_{p^{\ast}}=\hbox{ker }p^{\ast}(T)^k$. Then $\W_{p^{\ast}}$ can be considered as dual to $\W_p$ and the dual pairing gives a $T$-invariant non-degenerate form $h$ on $\W_p \oplus \W_{p^{\ast}}$, where $h|_{\W_p}=0=h|_{\W_{p^{\ast}}}$. 

 Suppose, $p(x)^k=(x-1)^k$.  Then the respective forms are obtained from \lemref{cuni}.  Suppose $p(x)^k=(x+1)^k$. Let $T_w$ denote the restriction of $T$ to $\W$. Then the minimal polynomial of $T_w$ is $(x+1)^k$. Thus the minimal polynomial of $-T_w$ is $(x-1)^k$. Further $T_w$ preserves a hermitian (resp. skew-hermitian) form if and only if $-T_w$ also preserves it. Thus this case reduces to the previous case and the existence of the required forms are clear. 

This completes the proof. \end{proof}

\begin{ack}
The second-named author gratefully acknowledges a UGC non-NET {fellowship} from Jadavpur University. He is highly grateful to his Ph.D advisor {Professor }Sujit Kumar Sardar for constant guidance, support, kindness and encouragements. He is also grateful to {Professor }K. C. Chattopadhyay for everything in life that has transformed him into a student of Mathematics. He is also thankful to Dr. Pralay Chatterjee and Dr. R. N. Mukherjee for their encouragements and help at various stages of life. 
\end{ack}

\end{document}